\documentclass{amsart}

\usepackage{helvet, color,float}

\usepackage{comment,amscd,amsmath,amsxtra,amsthm,amssymb,stmaryrd,xr,mathrsfs,mathtools,enumerate,url,xcolor}
\usepackage[all,cmtip]{xy}

\newtheorem{theorem}{Theorem}[section]
\newtheorem{lemma}[theorem]{Lemma}

\newtheorem{proposition}[theorem]{Proposition}

\newtheorem{definition}[theorem]{Definition}

\numberwithin{equation}{section}

\theoremstyle{remark}
\newtheorem{remark}[theorem]{Remark}
\newtheorem{example}[theorem]{Example}

    \theoremstyle{plain}

    \newtheoremstyle{TheoremNum}
        {\topsep}{\topsep}              
        {\itshape}                      
        {}                              
        {\bfseries}                     
        {.}                             
        { }                             
        {\thmname{#1}\thmnote{ \bfseries #3}}
    \theoremstyle{TheoremNum}
    \newtheorem{thmn}{Theorem}

\newcommand{\ord}{\mathrm{ord}}

\newcommand{\Fp}{\mathbb{F}_p}

\newcommand{\F}{\mathbb{F}}

\newcommand{\N}{\mathbb{N}}

\begin{document}

\title[Commuting polynomials]{Polynomials in $\Fp[x]$ which commute under composition}

\author[J.~Hatley]{Jeffrey Hatley}
\address[Hatley]{
Department of Mathematics\\
Union College\\
Bailey Hall 202\\
Schenectady, NY 12308\\
USA}
\email{hatleyj@union.edu}

\author[M.~Teplitskiy]{Mayah Teplitskiy}
\address[Teplitskiy]{
Department of Mathematics\\
Union College\\
Bailey Hall 202\\
Schenectady, NY 12308\\
USA}
\email{teplitsm@union.edu}

\begin{abstract} Let $\mathbb{F}$ be a finite field and let $f$ be a linear polynomial in $\mathbb{F}[x]$. We investigate the number of polynomials of degree $d$ which commute with $f$ under composition. In so doing, we rediscover a result of Park  \cite{Park-HG}, but with a conceptually simpler proof.
\end{abstract}



\maketitle

\section{Introduction}\label{section:intro}
In \cite{Karl}, Zimmerman studies the following question: given an algebraically closed field $\mathbb{F}$ of characteristic zero, a nonlinear polynomial $f \in \mathbb{F}[x]$, and a natural number $d$, how many polynomials in $\mathbb{F}[x]$ of degree $d$ commute with $f$ under composition? That is, how many polynomials $g \in \mathbb{F}[x]$ of degree $d$ satisfy $f \circ g = g \circ f$? This question was motivated by similar work by Julia \cite{Julia} and Ritt \cite{Ritt} regarding complex \textit{rational} functions.

The surprising main result of \cite{Karl} is that, for every $d \geq 1$, the number of polynomials in $\mathbb{F}[x]$ of degree $d$ which commute with $f$ is constant. The arguments used to establish this theorem rely crucially on the fact that $\mathbb{F}$ is algebraically closed.

Motivated by the work of Zimmerman, we investigated a similar question in a slightly different setting. More specifically, let $p$ be a prime, and let $\Fp$ be the finite field with $p$ elements. Let $f \in \Fp[x]$ be a linear polynomial (other than the ``identity'' polynomial $I(x)=x$). Given $d \geq 1$, we  sought to determine how many polynomials in $\Fp[x]$ of degree $d$ commute with $f$. Denote this number by $\#C_d(f)$. We offer proofs of the following results (but see $\S$ \ref{subsec:relation} for the history of such results).

\vspace{1ex}

\begin{thmn}[\ref{thm:main-nonmonic}]
Let $p \geq 3$ be a prime, and let $f \in \Fp[x]$ be a linear polynomial with leading coefficient $a \neq 1$. Let $d \geq 0$ be an integer, and let $r=\mathrm{ord}_p(a)$. Then
\[
\#C_d(f) = \begin{cases}
(p-1)p^{\frac{d-1}{r}},&\ \text{if}\ d \equiv 1 \mod{r}\\
0,&\ \text{if}\ d \not\equiv 1\mod{r}
\end{cases}.
\]
\end{thmn}

\begin{thmn}[\ref{thm:main-monic}]
Let $p \geq 3$ be a prime, let $f \in \Fp[x]$ be a monic linear polynomial, and let $d \geq 0$ be an integer. Then
\[
\# C_d(f) = \begin{cases}
(p-1)p^k,&\  \text{if}\ d=kp \ \text{for some integer}\ 0 \leq k \leq p\\
0,&\ \text{otherwise}
\end{cases}.
\]
\end{thmn}

\begin{remark}\label{remark:finite-field-extensions}
The careful reader will observe that our arguments go through verbatim upon replacing $\Fp$ with a finite field of order $q=p^m$ for any $m \geq 1$ and replacing each instance of $p$ with $q$ in the formulas given in Theorems \ref{thm:main-nonmonic} and \ref{thm:main-monic}.

In fact, the \textit{constructive} portions of our arguments would remain valid over the algebraic closure of $\Fp$, although the counting formulas would of course cease to be relevant.
\end{remark}

\subsection{Relation to previous work}\label{subsec:relation}
Over a decade earlier than \cite{Karl}, Park \cite{Park-HG} studied the following problem: For a finite field $\mathbb{F}$ and fixed $a,c \in  \mathbb{F}$, determine which polynomials $g \in \mathbb{F}_{p^k}[x]$ satisfy the relation $g(x+a)=g(x)+c$.

The authors were not aware of Park's work until after writing this paper. Clearly, taking $a=c$ in Park's work recovers the problem addressed by Theorem \ref{thm:main-monic}. We also note other similarities between our own work and Park's: the arguments in \cite{Park-HG} exploit explicit binomial coefficient identities modulo $p$, and the main result requires the degree bound $d < p^2$.

Nevertheless, we believe that our approach contains some significant conceptual simplifications. In particular, the key insights of this paper are:
\begin{itemize}
\item the observations in Section \ref{sec:algebraic-structure}, which allow us to focus our attention on particularly ``nice'' linear polynomials $f$, and
\item the recognition of an orderly, ``scaffolded'' structure, described via the \textit{orbit} sets $T_k$ defined in Section \ref{sec:constant-term-nonzero}, which is amenable to the application of useful mod $p$ binomial identities as recalled in Section \ref{sec:binomial}.
\end{itemize}
Thus, despite the many similarities between the present work and \cite{Park-HG}, we believe that our approach, discovered independently, gives a conceptually simpler and cleaner proof, while still remaining elementary and constructive.

\begin{remark}
In fact, Zieve and Masuda published a stronger result than ours or Park's only a year after Zimmerman's paper was published. Both Theorem \ref{thm:main-nonmonic} and Theorem \ref{thm:main-monic} (with no upper bound on the degree) can be obtained as special cases of \cite[Theorem 1.1]{ZieveMasuda}. 
The authors were unaware of \cite{ZieveMasuda} until after the completion of this paper.
\end{remark}

\subsection*{Acknowledgements}\label{sec:acknowledgements}

The authors are thankful to Michael Zieve for bringing their attention to his work with Ariane Masuda, and also to the work of Hong Goo Park.

The authors thank Union College for providing an encouraging environment for undergraduate research.

\section{Hypotheses and notation}\label{sec:notation}

Throughout this paper we let $p$ denote a fixed odd prime. Let $\Fp[x]$ denote the ring of polynomials with coefficients in $\Fp$. As all of our computations will take place over the base field $\Fp$, we will often write $=$ instead of $\equiv$ mod $p$.

Let $f, g \in \Fp[x]$. If $(f \circ g) = (g \circ f)$, then we write $f \sim g$. It is easy to see that this is an equivalence relation on $\Fp[x]$.

Let $f \in \Fp[x]$ and let $d \geq 1$ be a natural number. We write
\[
C_d(f) = \{ g \in \Fp[x] \colon f \sim g  \ \text{and} \ \mathrm{deg}(g)=d\}.
\]

\section{Algebraic structure of $C_d(f)$}\label{sec:algebraic-structure}

In this section, we record some useful lemmas on the algebraic structure of the sets $C_d(f)$. In particular, we are able to reduce the problem of determining $C_d(f)$ for general $f \in \Fp[x]$ to just two cases: $f=x+1$ and $f=ax$ for $a\neq 1$.

\subsection{Similar polynomials}\label{sec:similar-polys}
Let $\lambda=\alpha x + \beta \in \Fp[x]$ be a linear polynomial, so in particular $\alpha \in \Fp^\times$. We denote its inverse by $\lambda^{-1} = \alpha^{-1}x - \alpha^{-1}\beta \in \Fp[x]$.

We will make use of the following useful notion which also plays a key role in \cite{Karl}.

\begin{definition}\label{def:similar}
Let $f,g \in \Fp[x]$. We say $f$ and $g$ are \underline{similar} if there exists a linear polynomial $\lambda \in \Fp[x]$ such that
\[
g = \lambda^{-1} \circ f \circ \lambda.
\]
\end{definition}

One easily checks that if $f$ and $g$ are similar then they have the same degree. In fact, we have the following useful results.

\begin{lemma}\label{lem:similarity-classes}
Let $f=ax+b \in \Fp[x]$ be a linear polynomial, so in particular $a \in \Fp^\times$. Furthermore, if $b=0$, then assume $a \neq 1$.
\begin{enumerate}
\item If $a \neq 1$, then $f$ is similar to $ax+c$ for any $c \in \Fp$.
\item If $a=1$ (so necessarily $b \neq 0)$, then $f$ is similar to $x+c$ for any $c \in \Fp^\times$.
\end{enumerate}
\end{lemma}
\begin{proof}
To prove (1) (resp. (2)), just take $\lambda$ to be $x+(c-b)(a-1)^{-1}$ (resp. $cb^{-1}$).
\end{proof}

\begin{proposition}\label{prop:conjugate}
Let $f,g \in \Fp[x]$ and let $d \in \N$. Suppose $f$ is similar to $g$ by $\lambda.$ Then $\lambda$ induces an isomorphism of sets
\[
C_d(f) \xrightarrow{\sim} C_d(g).
\]
\begin{proof}
We define the map
\[
\Lambda \colon C_d(f) \rightarrow C_d(g)
\]
by
\[
\Lambda (P) = \lambda^{-1} \circ P \circ \lambda.
\]
Since this map is clearly invertible, it suffices to show that the map is well-defined. Thus, let $P \in C_d(f)$ and let $Q=\Lambda(P)$. We must show $Q \in C_d(g)$.
Suppose $P \in C_d(f)$, and let $Q=\lambda^{-1} \circ P \circ \lambda$. Then
\begin{align*}
g \circ Q &= (\lambda^{-1}\circ f \circ \lambda) \circ (\lambda^{-1} \circ P \circ \lambda)\\
&= \lambda^{-1} \circ (f \circ P) \circ \lambda \\
&=\lambda^{-1} \circ (P \circ f) \circ \lambda &\text{since}\ P \in C_d(f)\\
&=(\lambda^{-1} \circ P \circ \lambda) \circ (\lambda^{-1} f \circ \lambda)\\
&= Q \circ g.
\end{align*}
Thus $Q \in C_d(g)$ as desired.
\end{proof}
\end{proposition}

\begin{remark}\label{remark:simplification} In particular, to prove our main theorems, it suffices to study $C_d(x+1)$ and $C_d(ax)$ for $a\notin\{0,1\}$.
\end{remark}

\subsection{Group structure of $C_d(x+1)$}\label{sec:group-structure}

In this section, we observe that $C_d(x+1)$ may be given the structure of a group. For the rest of this section, write $f=x+1$.

\begin{definition}\label{def:oplus}
Let $g,h \in \Fp[x]$. Define the polynomial $g \oplus h$ by
\[
(g \oplus h)(x) = g(x) + h(x) - x.
\]
\end{definition}

Note that this definition does not require $g$ and $h$ to have the same degree.

\begin{proposition}\label{prop:commuters-closed-under-oplus}
Let $f=x+1$ and let $g,h \in \Fp[x]$. If $f \sim g$ and $f \sim h$, then $f \sim (g \oplus h)$.
\end{proposition}

\begin{proof}
We compute
\begin{align*}
    [(g \oplus h) \circ f](x) &= g(x+1) + h(x+1) - (x+1)\\
    &= [g(x)+1] + [h(x)+1] -x - 1\\
    &= g(x) + h(x) - x + 1\\
    &=[f \circ (g \oplus h)](x).
\end{align*}
\end{proof}

It is straightforward to check that this makes $(C_d(x+1),\oplus)$ into a group with identity element $x$. In particular, the linear term can be seen as a ``correction factor'' among polynomials of $C_d(x+1)$. In contrast, the constant term is totally free.

\begin{lemma}\label{lem:constant-free}
Suppose $g \in \Fp[x]$ and $f \sim g$. Then for any $c \in \Fp$, we also have $f \sim (g+c)$.
\end{lemma}
\begin{proof}
We compute
\begin{align*}
    [(g+c) \circ f](x) &= g(x+1) + c\\
    &= [g(x)+1] + c\\
    &= [g(x) + c]+ 1\\
    &=[f \circ (g +c)](x).
\end{align*}
\end{proof}

\section{Two useful binomial coefficient identities}\label{sec:binomial}

Recall the binomial expansion
\[
(ax+b)^n = \sum_{i=0}^n \binom{n}{i} (ax)^{n-i} b^i.
\]
In this section, we collect some useful identities for the binomial coefficients $\binom{n}{i}$ modulo $p$. For instance, in an elementary number theory course, one encounters the so-called ``Freshman's Dream''
\[
(x+y)^p \equiv x^p + y^p \mod p,
\]
which is equivalent to the statement that
\[
\binom{p}{i} \equiv 0 \mod p, \quad 1 \leq i \leq p-1.
\]
The following generalization of this fact is due to Lucas \cite[Section XXI]{Lucas}.

\begin{theorem}[Lucas]\label{thm:Lucas}
Let $p$ be a prime number, and let $m,n$ be nonnegative integers. Denote by
\[
m=\sum_{i=0}^k m_i p^i
\]
the base-$p$ expansion of $m$, and similarly for $n$. Then
\[
\binom{m}{n} \equiv \prod_{i=0}^k \binom{m_i}{n_i} \mod p.
\]
In particular, $n_i > m_i$ for some $1 \leq i \leq k$ if and only if  $\binom{m}{n} \equiv 0 \mod p$.
\end{theorem}

At one crucial moment we will also need the following well-known identity.

\begin{theorem}[Chu-Vandermonde Identity]\label{thm:chu-vand}
Let $m,n,$ and $k$ be non-negative integers. Then
\[
\sum_{j=0}^k \binom{m}{j} \binom{n-m}{k-j} = \binom{n}{k}.
\]
\end{theorem}

\section{Main results for non-monic linear polynomials }\label{sec:constant-term-zero}

In this section we will prove the following theorem.

\begin{theorem}\label{thm:main-nonmonic}
Let $f \in \Fp[x]$ be a linear polynomial with leading coefficient $a \neq 1$. Let $d \geq 0$ be an integer, and let $r=\mathrm{ord}_p(a)$. Then
\[
\#C_d(f) = \begin{cases}
(p-1)p^{\frac{d-1}{r}},&\ \text{if}\ d \equiv 1 \mod{r}\\
0,&\ \text{if}\ d \not\equiv 1\mod{r}
\end{cases}.
\]
\end{theorem}

Recall that, by Remark \ref{remark:simplification}, it suffices to consider the case $f(x)=ax$ for $a \neq 1$.

\begin{proof}
Let $f(x)=ax$ for $a \neq 1$, and let $g \in \F_p[x]$ be a polynomial of degree $d$. Then
\begin{align*}
    f(g(x))&=a(c_dx^d+c_{d-1}x^{d-1}+...+c_0), \quad \text{and}\\
    g(f(x))&=c_d(ax)^d+c_{d-1}(ax)^{d-1}+...+c_0.
\end{align*}
Assume $f \sim g$. Then for every $0 \leq i \leq d$, we must have
\begin{equation}\label{eq:ord-divides-degree}
    ac_ix^i \equiv a^i c_i x^i \mod p.
\end{equation}
In particular, since $g$ is degree $d$, we have $c_d\neq 0$, so \eqref{eq:ord-divides-degree} implies that

\begin{equation}\label{eq:ord-divides-degree-case-d}
    a \equiv a^d  \mod p.
\end{equation}
Thus, it is necessary that $r=\ord_p(a) \mid d-1$, in which case \eqref{eq:ord-divides-degree-case-d} is satisfied for any of the $p-1$ choices of $c_d \in \Fp^\times$.

Now, for any $0 \leq i \leq d$, the same argument shows that if $c_i \neq 0$, then $r \mid i-1$. Since we must have $r \mid d-1$, there are only $\frac{d-1}{r}$ such indices for which we may have $c_i \neq 0$, and in each such case we have $p$ choices for $c_i$. This completes the proof.
\end{proof}

\section{Main results for monic linear polynomials}\label{sec:constant-term-nonzero}

We now turn our attention to the case of monic linear polynomials, where we have the following result.

\begin{theorem}\label{thm:main-monic}
Let $p\geq 3$ be a prime, let $f \in \Fp[x]$ be a monic linear polynomial, and let $d \geq 0$ be an integer. Then
\[
\# C_d(f) = \begin{cases}
(p-1)p^k,&\  \text{if}\ d=kp \ \text{for some integer}\  0\leq k \leq p\\
0,&\ \text{otherwise}
\end{cases}.
\]
\end{theorem}

We begin by proving that $f$ can only commute with a polynomial of degree $kp$.
\begin{proposition}\label{prop:p-divides-d}
Let $f$ be a monic linear polynomial and let $g$ be any polynomial of degree $d$ that is not the identity. If $f(x) \sim g(x)$, then $d=kp$ for some $k \in \N$.
\end{proposition}
\begin{proof}
Recall that it suffices to consider the case where $f(x)=x+1$. Assume $f(x) \sim g(x)$. Then comparing $(d-1)$st coefficients of $f(g(x))$ and $g(f(x))$, we see
\[
c_{d-1} \equiv c_{d-1}+dc_d \mod{p}
\]
This will only hold if $d \equiv 0\mod{p}$. Hence, $d=kp$ for some $k \in \N$.
\end{proof}

Define
\[
T_0=\{0\} \quad \text{and} \quad T_{p}=\{p^2\},
\]
and for $0 < k < p$, inductively define
\begin{align*}
T_k &= \{i+1 \ \vert \ i \in T_{k-1}\} \cup \{kp\} \\
&= \{ip +j \ \vert \ i,j \geq 0, \  i + j = k\}.
\end{align*}
Then set
\[
R_k=\bigcup_{i=0}^k T_k \quad \text{and} \quad S_k=\{ 0,1,2,\ldots, kp \} \setminus R_k.
\]
Thus, when $0 \leq k \leq p$, any polynomial $g$ of degree $kp$ may be written in the form
\[
g=\sum_{I \in R_k} c_I x^I + \sum_{J \in S_k} c_J x^J.
\]

\begin{remark}\label{rem:ordered-pairs} From the second description of $T_k$ above, we may identify each element of $T_k$ with an ordered pair $(i,j)$.
\end{remark}

\begin{example}
If $p=5$, the sets obtained in this fashion are shown in the following table. 
\begin{table}[H]
\begin{tabular}{l|l|l|l}
$k$ & $T_k$            & $R_k$                                            & $S_k$                            \\ \hline
0   & \{0\}            & \{0\}                                            &                                  \\
1   & \{1,5\}          & \{0,1,5\}                                        & \{2,3,4\}                        \\
2   & \{2,6,10\}       & \{0,1,2,5,6,10\}                                 & \{3,4,7,8,9\}                    \\
3   & \{3,7,11,15\}    & \{0,1,2,3,5,6,7,10,11,15\}                       & \{4,8,9,12,13,14\}               \\
4   & \{4,8,12,16,20\} & \{0,1,2,3,4,5,6,7,8,10,11,12,15,16,20\}          & \{9,13,14,17,18,19\}             \\
5   & \{25\}           & \{0,1,2,3,4,5,6,7,8,10,11,12,15,16,20,25\}       & \{9,13,14,21,22,23,24\}
\end{tabular}
\end{table}
\end{example}

For any integer $k \geq 0$, let $[kp]$ denote the set of integers between $(k-1)p$ and $kp$, inclusive. Set $\tilde{R}_k=R_k \cap [kp]$ and $\tilde{S}_k=S_k \cap [kp]$. Then we see that
\begin{equation}\label{interval-description-rk}
\tilde{R}_k = \begin{cases}
\{ (k-1)p, (k-1)p +1, kp \} & \text{if}\ p \nmid k \\
\{ (k-1)p,  kp \} & \text{if}\ p \mid k
\end{cases}
\end{equation}
and
\begin{equation}\label{interval-description-sk}
\tilde{S}_k = \begin{cases}
\{ (k-1)p+i \ | \ 2 \leq i \leq p-1\} & \text{if}\ p \nmid k \\
\{ (k-1)p+i \ | \ 1 \leq i \leq p-1\} & \text{if}\ p \mid k
\end{cases}.
\end{equation}

The following lemma sheds light on the definitions above, and it is crucial to our approach.

\begin{lemma}\label{lem:lucas-for-coefficient-sets}
Let $k \geq 0$ be an integer. Let $r \in R_k$ and $s \in S_k$. Then
\[
\binom{r}{s-1} = 0.
\]
\end{lemma}
\begin{proof}
We may assume without loss of generality that $s \in \tilde{S}_k$, i.e. that $(k-1)p \leq s-1 \leq kp-2$. Since the conclusion of the lemma is trivial if $r<s-1$, we may assume $r > s-1$, so by \eqref{interval-description-rk} we must have $r=kp$. The result now follows from Theorem \ref{thm:Lucas}.
\end{proof}

As a corollary, we obtain the following important description of polynomials in $C_{kp}(x+1)$.
\begin{lemma} \label{lem:cascade}
Let $g \in C_{kp}(x+1)$. Then $g$ is of the form
\[
g = \sum_{I \in R_k} c_I x^I.
\]
That is, $c_J=0$ for each $J \in S_k$.
\end{lemma}

\begin{proof}
Write
\[
g=\sum_{I \in R_k} c_I x^I + \sum_{J \in S_k} c_J x^J
\]
and suppose $n$ is the largest value in $S_k$ for which $c_n \neq 0$.
Since $g \in C_{kp}(x+1)$, we have
\begin{equation}\label{eq:g-commutes}
g(x)+1=g(x+1).
\end{equation}
The coefficient of $x^{n-1}$ on the left-hand side is clearly $c_{n-1}$. Let us also determine this coefficient on the right-hand side.

Only terms of index greater than or equal to $n-1$ will contribute to this coefficient. If $m > n$, then $c_m \neq 0$ implies $m \in R_k$, and so the coefficient of $x^{n-1}$ in the expansion of $c_m(x+1)^m$ vanishes by Lemma \ref{lem:lucas-for-coefficient-sets}. Thus, the coefficient of $x^{n-1}$ on the right-hand side of \eqref{eq:g-commutes} is
\[
 c_{n-1}+nc_n.
\]
Since $n \in S_k$, we know by \eqref{interval-description-sk} that $n \not\equiv 0 \mod p$, so this implies that $c_n=0$. This contradicts our choice of $n$, hence
\begin{equation}\label{eq:gk}
g = \sum_{i \in R_k} c_i x^i
\end{equation}
as desired.
\end{proof}

We will prove Theorem \ref{thm:main-monic} by determining the dependencies between the various coefficients of $g$. To begin with, we have the following very simple lemma.

\begin{lemma}\label{lem:constant-term-free}
If $g(x) \in C_d(x+1)$ then $g(x)+c \in C_d(x+1)$ for any $c \in \Fp$.
\end{lemma}
\begin{proof}
If $g(x)+1=g(x+1)$, then
\[
g(x+1)+c=[g(x)+1]+c=[g(x)+c]+1.
\]
\end{proof}


In the argument that follows, we say that coefficients (identified by their indices) occupy the same \textit{orbit} if the choice of one determines the values of the others. So for instance, by Lemma \ref{lem:constant-term-free} we always have the singleton orbit $\{0\}$, while Example~\ref{example:case-k-1} below will establish that in degree $p$ we have an additional orbit $\{1,p\}$.
Every orbit gives rise to $p$ possible choices (or $p-1$ for the orbit of the top-degree term), so determining how to partition $R_k$ into orbits gives us the size of $C_{kp}(x+1)$: namely, if $R_k$ can be partitioned into $\ell+1$ orbits, then $\#C_{kp}(x_1)=(p-1)p^\ell$. We will show that the orbits at degree $kp$ are precisely the sets $T_0,\ldots,T_k$, from which Theorem \ref{thm:main-monic} follows.

Since we have an addition operation on $C_{kp}$ as described in Section \ref{sec:algebraic-structure}, it suffices to show that we may construct an element of $C_{kp}(x+1)$ whose only nonzero coefficients come from $T_k$ plus a linear term (which, in light of Section \ref{sec:algebraic-structure}, can be viewed as a correction term), and that these coefficients are all determined by the choice of the leading coefficient.

Before proving this last claim, let us illustrate our argument in the cases $k=1$ and $k=2$.
\begin{example}\label{example:case-k-1}
\underline{Case $k=1$}

By Lemma \ref{lem:cascade}, if $g\in C_{p}(x+1)$, it may only have nonzero coefficients at the degrees in $R_1=\{0,1,p\}$. By Lemma \ref{lem:constant-term-free}, it suffices to consider
\[
g(x)=c_p x^p + c_1 x.
\]
Then
\begin{align*}
g(x+1) &= c_p x^p + c_1 x + (c_1 + c_p)\\
&=g(x)+(c_1 + c_p),
\end{align*}
which shows that we are free to pick any value for $c_p$, and then choosing $c_1 = 1-c_p$ yields an element of $C_p(x+1)$. There are $p$ choices for $c_0$ and $p-1$ nonzero choices for $c_p$, giving $(p-1)p$ polynomials in $C_p(x+1)$.
\end{example}
\begin{example}
\underline{Case $k=2$}

By Lemma \ref{lem:cascade}, if $g\in C_{2p}(x+1)$, it may only have nonzero coefficients at the degrees in $R_2=\{0,1,p,2,p+1,2p\}$. The degrees which were not present in the case $k=1$ are precisely those in $T_2=\{ 2,p+1,2p\}$. Any polynomial we can construct from just $T_2$ may be added (in the sense of Definition \ref{def:oplus}) to polynomials constructed from $T_0$ and $T_1$ to yield more polynomials in $C_{2p}$. So let us determine how many such polynomials we can build.

Suppose
\[
g(x)= c_{2p} x^{2p} + c_{p+1}x^{p+1} + c_2 x^2 + x.
\]
Then
\[
g(x+1) = g(x) + (2 c_{2p} + c_{p+1})x^p + (c_{p+1} + 2c_2)x + (c_{2p}+c_{p+1}+c_2+1).
\]
Thus, upon choosing a value for $c_{2p}$, we must have $c_{p+1}=-2c_p$ and $c_2=c_{2p}$, and then the constant term in the displayed equation is identically 1, as desired. Thus, the indices in $T_2$ form a new, distinct orbit.

Excluding zero, there are $(p-1)$ choices for $c_{2p}$, so there are $(p-1)$ choices of $g(x)$ as above. Under our addition law for $C_d(x+1)$ [$g_1 \oplus g_2 = g_1 + g_2 - x$], we may combine these choices of $g(x)$ with those of degree $p$. Since we may now allow $c_p=0$, there are $p^2$ polynomials in $C_{2p}(x+1)$ with which to combine, so these combinations give us $\#C_{2p}(x+1)=(p-1)p^2$ polynomials in $C_{2p}(x+1).$
\end{example}


As our arguments above have explained, it suffices to show that for each $k$ we can construct a polynomial in $C_{kp}$ using only indices from $T_k$. We do this now.

Suppose
\[
g(x)= x + \sum_{I \in T_k} c_I x^I.
\]
Using Theorem \ref{thm:Lucas}, we find that $g(x+1)=g(x)+h(x)$ where
\begin{equation}\label{eq:expanded-binomial}
h(x)  =  (1 + \sum_{I \in T_k}c_I) +  \sum_{t \in R_{k-1}} \left( \sum_{I \in T_k} \binom{I}{t} c_{I}\right)x^{t}
\end{equation}

We now make a sequence of observations about the coefficients  appearing on the right-hand side of \eqref{eq:expanded-binomial} in order to show that, for each choice of value for $c_{kp}$, there is a unique choice of value for each other $c_I$, $I \in T_k$ which makes $h(x)$ identically zero.

First note that the case $k=p$ is special, since one immediately checks that $x^{p^2}$ commutes with $x+1$. Thus, $[p^2]$ is its own orbit. We now consider $k<p$.

Since $k<p$, we have
\[
T_k = \{ ip + j \ | \ i,j  \geq 0, \ i+j=k\}.
\]
We identify the element $ip+j \in T_k$ with the ordered pair $(i,j)$.

Let $t \in T_{k-1}$, so that $t=i' p + j'$ where $i' + j' = k-1$. Then from \eqref{eq:expanded-binomial}, the coefficient of $x^t$ is
\begin{align*}
\sum_{I=(i,j) \in T_k} \binom{I}{t} c_{I} & = \sum_{\substack{i,j \geq 0 \\ i+j=k}} \binom{ip+j}{i' p + j'} c_{(i,j)} \\
&=\sum_{\substack{i,j \geq 0 \\ i+j=k}} \binom{i}{i'} \binom{j}{j'} c_{(i,j)} & \text{by Theorem \ref{thm:Lucas}}\\
\end{align*}
Applying Theorem \ref{thm:Lucas} again, we see that the only binomial coefficients in the sum above which are not zero in $\Fp$ are
\[
\binom{i'+1}{i'}\binom{j'}{j'}=i'+1 \quad \quad  \text{and} \quad \quad \binom{i'}{i'}\binom{j'+1}{j'}=j'+1.
\]
Thus, the coefficient of $x^t$ is
\[
(i'+1) c_{(i'+1,j')}  + (j'+1) c_{(i',j'+1)},
\]
and if $g \sim f$ then this coefficient must be zero. This produces a chain of dependencies between the $c_I$ for $I \in T_k$:
\begin{align*}
(k)c_{(k,0)} &= -c_{(k-1,1)}\\
(k-1)c_{(k-1,1)} &= - 2c_{(k-2,2)}\\
(k-2)c_{(k-2,2)} &= - 3c_{(k-3,3)}\\
& \vdots \\
c_{(1,k-1)}&=-kc_{(0,k)}.
\end{align*}
In particular, choosing a value for $c_{(k,0)}=c_{kp}$ determines the values for the rest of the $c_I$ for $I \in T_k$, and these choices are consistent with the requirement from the constant term in \eqref{eq:expanded-binomial} that $\sum_{I \in T_k} c_I =0$.

The only thing left to check is that the conditions imposed by the coefficients of $x^t$ for $t \in R_{k-2}$ do not contradict the ones we just found.
Let $1 \leq k' \leq k-2$. Using the same style of computation as above, the coefficient of $x^t$ for each $t \in T_{k'}$ yields a linear dependence equation between the $c_I$ for $I \in T_k$. Rather than write these individually, consider the sum of these linear dependence equations. In this sum, the coefficient of $c_{(i,j)}$ is
\begin{align*}
\sum_{\substack{i',j' \geq 0 \\ i' + j' = k'}} \binom{ip+j}{i'p+j'} &= \sum_{\substack{i',j' \geq 0 \\ i' + j' = k'}} \binom{i}{i'} \binom{j}{j'} \\
&= \binom{k}{k'} & \text{by Theorem \ref{thm:chu-vand}.}
\end{align*}
Since this is true for every $I \in T_k$, we see that the linear dependence conditions imposed by $R_{k-2}$ are also consistent with the condition imposed by the linear term in \eqref{eq:expanded-binomial}. This completes the proof.

\section{Some numerical examples}
We now provide examples which illustrate the methods in Section \ref{sec:constant-term-nonzero}.

\underline{Example \#1:}

Let $p=3$ and $k=2$. By Theorem \ref{thm:main-monic}, we calculate $18$ polynomials that commute with $x+1$. We will now show why this is the case in this particular example. By Lemma 6.3, we know that this polynomial will take the form
$c_6x^6+c_4x^4+c_3x^3+c_2x^2+c_1x+c_0$. We also know that our orbits will be $[0],[1,3],[2,4,6]$.

\begin{center}
\begin{table}[H]
\begin{tabular}{|l|l|l|l|l|l|l|}
\hline
determined degree & $c_6$ & $c_4$ & $c_3$ & $c_2$ & $c_1$ & $c_0$ \\ \hline
$c_6$=1           & $c_6$=1         & $c_4$=1         &             & $c_2$=1         &             &             \\ \hline
$c_6$=2           & $c_6$=2         & $c_4$=2         &             & $c_2$=2         &             &             \\ \hline
$c_3$=0           &             &             & $c_3$=0         &             & $c_1$=1         &             \\ \hline
$c_3$=1           &             &             & $c_3$=1         &             & $c_1$=0         &             \\ \hline
$c_3$=2           &             &             & $c_3$=2         &             & $c_1$=2         &             \\ \hline
$c_0=0$           &             &             &             &             &             & $c_0=0$         \\ \hline
$c_0=1$           &             &             &             &             &             & $c_0=1$         \\ \hline
$c_0=2$           &             &             &             &             &             & $c_0=2$         \\ \hline
\end{tabular}
\end{table}
\end{center}
We can see from the table that there are $18$ possible polynomials of degree $6$ that commute with $x+1$. We can build these polynomials by selecting one choice for $c_6,c_3$ and $c_0$ and combining the three. Some examples of these polynomials are
\begin{itemize}
    \item $x^6+x^4+x^2 + x$
        \item $x^6+x^4+x^3 + x^2 $
    \item $2x^6+2x^4+x^3+2x^2+1$
    \item $x^6+x^4+2x^3+x^2+2x+2$
\end{itemize}

\underline{Example \#2:}

Let $p=5$ and $k=5$. By using our method for finding degrees that can show up in $g$, we see that degrees 0,1,2,3,4,5,6,7,8,10,11,12,15,16,20,25 can appear in $g$, while all other degrees $0 \leq d \leq 25$ cannot.

Then we can see that our orbits have the following structure:
\begin{center}
\begin{table}[H]
\begin{tabular}{l|l|l}
$k$ & new degrees  & orbits                                                              \\ \hline
1   & 0,1,5        & {[}0{]}, {[}1,5{]}                                                  \\
2   & 2,6,10       & {[}0{]},{[}1,5{]},{[}2,6,10{]}                                      \\
3   & 3,7,11,15    & {[}0{]},{[}1,5{]},{[}2,6,10{]}, {[}3,7,11,15{]}                     \\
4   & 4,8,12,16,20 & {[}0{]},{[}1,5{]},{[}2,6,10{]}, {[}3,7,11,15{]}, {[}4,8,12,16,20{]}\\
5   &25            & {[}0{]},{[}1,5{]},{[}2,6,10{]}, {[}3,7,11,15{]}, {[}4,8,12,16,20{]}, {[}25{]}
\end{tabular}
\end{table}
\end{center}

\bibliographystyle{amsalpha}
\bibliography{references}

\providecommand{\bysame}{\leavevmode\hbox to3em{\hrulefill}\thinspace}
\providecommand{\MR}{\relax\ifhmode\unskip\space\fi MR }
\providecommand{\MRhref}[2]{%
  \href{http://www.ams.org/mathscinet-getitem?mr=#1}{#2}
}
\providecommand{\href}[2]{#2}
\begin{thebibliography}{Zim07}

\bibitem[Jul22]{Julia}
Gaston Julia, \emph{M\'{e}moire sur la permutabilit\'{e} des fractions
  rationnelles}, Ann. Sci. \'{E}cole Norm. Sup. (3) \textbf{39} (1922),
  131--215. \MR{1509242}

\bibitem[Luc78]{Lucas}
Edouard Lucas, \emph{Theorie des {F}onctions {N}umeriques {S}implement
  {P}eriodiques. [{C}ontinued]}, Amer. J. Math. \textbf{1} (1878), no.~3,
  197--240. \MR{1505164}

\bibitem[MZ08]{ZieveMasuda}
Ariane~M. Masuda and Michael~E. Zieve, \emph{Rational functions with linear
  relations}, Proc. Amer. Math. Soc. \textbf{136} (2008), no.~4, 1403--1408.
  \MR{2367113}

\bibitem[Par92]{Park-HG}
Hong~Goo Park, \emph{Polynomials satisfying {$f(x+a)=f(x)+c$} over finite
  fields}, Bull. Korean Math. Soc. \textbf{29} (1992), no.~2, 277--283.
  \MR{1180621}

\bibitem[Rit23]{Ritt}
J.~F. Ritt, \emph{Permutable rational functions}, Trans. Amer. Math. Soc.
  \textbf{25} (1923), no.~3, 399--448. \MR{1501252}

\bibitem[Zim07]{Karl}
Karl Zimmermann, \emph{Commuting polynomials and self-similarity}, New York J.
  Math. \textbf{13} (2007), 89--96. \MR{2314218}

\end{thebibliography}
\end{document}